\documentclass[a4paper,12pt]{article}
\usepackage{xcolor}
\usepackage{amsmath}
\usepackage{amssymb}
\usepackage{latexsym}
\usepackage{subcaption}
\usepackage{graphicx}
\usepackage{amsthm}

\newtheorem{theorem}{Theorem}[section]
\newtheorem{lemma}[theorem]{Lemma}

\theoremstyle{definition}
\newtheorem{definition}[theorem]{Definition}
\newtheorem{example}[theorem]{Example}
\newtheorem{remark}[theorem]{Remark}
\numberwithin{equation}{section}

\title{\LARGE{
Green's function for singular fractional differential equations and applications
 }}

\begin{document}

\author{%
  Jinsil Lee${}^{a}$ and Yong-Hoon Lee${}^{b,1}$   \
       \\[1ex]
     {\small\itshape ${}^{a}$ Department of Mathematics, University of Georgia,} \\  
     {\small\itshape Athens, GA 30606, USA}\\ 
    {\small\itshape ${}^{b}$ Department of Mathematics, Pusan National University,} \\
    {\small\itshape Busan 46241, Republic of Korea}\\
     {\small\upshape E-mail: jl74942@uga.edu}\\
    {\small\upshape E-mail: yhlee@pusan.ac.kr}
    }

\footnotetext[1]{Corresponding Author}

\date{}
\maketitle

\begin{abstract}
In this paper, we study the existence of positive solutions for nonlinear fractional differential equation with a singular weight. We derive the Green's function and corresponding integral operator and then examine compactness of the operator. As an application, we prove an existence result for positive solutions when nonlinear term satisfies either superlinear or sublinear conditions.
The proof was mainly employed  by  Krasnoselski's classical fixed point theorem.



\end{abstract}
\section{Introduction}
Many problems in physics, control theory, and chemistry can be represented as fractional differential equations (\cite{A, B}). Particularly, many researchers have made contributions to  
the existence and multiplicity of positive solutions for nonlinear fractional differential equations using Krasnoselski's fixed point theorem (\cite{C, E, Q}).
We are interested in applying it to the following equations with a singular weight:

\begin{equation*}\tag*{$(FDE)$}\label{FDE}
\begin{cases}
 D^{\alpha}_{0+}u(t)+h(t)f(u(t))= 0,\quad t\in (0,1),\\
u(0)= 0 = u(1),
\end{cases}
\end{equation*}
where $D^{\alpha}_{0+}$ is the Riemann-Liouville fractional derivative of order $\alpha \in (1,2]$, 
$f\in C([0,\infty),[0,\infty)) $ is a given continuous function and $h\in L^{1}_{loc}((0,1),[0,\infty))$ satisfies the following conditions: 
\begin{flushleft}
$(H_1) \ \ \int_0^1 s^{\alpha -1}h(s)ds<\infty$,\\
$(H_2) \ \ \ h$ is bounded on any compact subinterval in $(0,1].$\\
\end{flushleft}
We notice that coefficient function $h$ satisfying condition $(H_1)$ may not be integrable near $t=0$, as an example, we may consider $h(t)=t^{-\beta}$ where $1<\beta <\alpha$. We see that $h$ satisfies conditions $(H_1)$ and $(H_2)$ but $h\notin L^{1}((0,1),[0,\infty))$.
  
Introducing the Green's function for the case that $h$ is continuous,
Bai and L\"u \cite{E} consider the following nonlinear problem 
\begin{align}\label{eq:::1.1}
\begin{cases}
 D^{\alpha}_{0+}u(t)+f(t,u(t))= 0,\quad t\in (0,1),\\
u(0)= 0 = u(1),
\end{cases}
\end{align}
 where 
$f\in C([0,1]\times[0,\infty), [0,\infty)$). By taking the Riemann-Liouville fractional integral, they set up an equivalent solution operator $S$ by
\begin{align}
Su(t)= \int_{0}^{1} G(t,s)f(s,u(s))ds
\end{align}
where $G(t,s)$ defined by
\begin{align}\label{eq:::1.2}
G(t,s)=
\begin{cases}\displaystyle
 \frac{(t(1-s))^{\alpha-1}-(t-s)^{\alpha-1}}{\Gamma(\alpha)}, \quad 0\leq s\leq t\leq 1,\\\displaystyle
 \frac{(t(1-s))^{\alpha-1}}{\Gamma(\alpha)}, \quad 0\leq t\leq s\leq 1
\end{cases}
\end{align}
is the Green's function for the fractional differential equation
$$ D^{\alpha}_{0+}u(t)= 0$$
with Dirichlet boundary condition. Analysing this operator, they proved the existence of at least three positive solution of problem \eqref{eq:::1.1} under some additional conditions on $f$.

In \cite{Q}, Jiang and Yuan studied positive solutions of problem \eqref{eq:::1.1} with the following hypotheses:\\

\noindent
(A) \ There exist $g\in C([0,\infty), [0,\infty)$), $q_{1},q_{2}\in C((0,1), (0,\infty)$) such that
$$q_{1}(t)g(u)\le f(t,t^{\alpha-2}u)\le q_{2}(t)g(u),$$
and $q_i, \ i=1,2$ also satisfy $q_{i} \in L^1 (0,1)$.\\

\noindent
Under condition (A), they proved that problem \eqref{eq:::1.1} has at least one positive solution either
\begin{itemize}
\item [(1)] \ $g_0=0$, $g_\infty=\infty$ or
\item [(2)] \ $g_0=\infty$, $g_\infty=0$,
\end{itemize}
where $g_0$ and $g_\infty$ are defined by $g_0=\lim_{u\to 0}\frac{g(u)}{u}$ and $g_\infty =\lim_{u\to \infty}\frac{g(u)}{u}$. 
Our concern in this paper is focused on the case that given weight function $h$ is singular at the boundary which may not be integrable on $(0,1)$, and nonlinear term $f$ satisfies conditions \begin{itemize}
\item [(A1)] \ $f_0=0$, $f_\infty=\infty$, and there exist $p$ satisfying $\lim_{u\to \infty}\frac{f(u)}{u^p}=0$ or
\item [(A2)] \ $f_0=\infty$, $f_\infty=0$.
\end{itemize}  In most works, function $h$ to be assumed integrable on $(0,1)$ so that the Green's function can be derived based on the integrability of $h$ near boundary $0$. Under this assumption, the unique solution of (1.1) is represented by the solution of an integral equation (1.2) using the fractional integral.
However, if $h$ is not integrable, we should consider the existence of $D^\alpha_{0+}u$ and its solution space. Moreover, corresponding Green's function can not be obtained by obvious modification from the case $h\in L^1 .$ In the paper \cite{N}, the researchers considered the existence of the solution for the second order differential equation where the function $f$ is a given function satisfying Caratheodory's conditions with singularities at 0 and 1. We are interested in extending the existence results to the fractional case. 
In \cite{LL} , Lee and Lee define the solution space and the definition of the solution and derive the Green's function in this singular situation. \begin{lemma}(\cite{LL})\label{lemma3.2}
Assume $g$ satisfies $(H_1)$ and $(H_2 )$, then the following equation 
\begin{equation*}\label{P_1}
\begin{cases}
 D^{\alpha}_{0+}u(t)+g(t)= 0,\quad t\in (0,1),\\
u(0)= 0 = u(1),
\end{cases}
\end{equation*}
is equivalent to the functional integral equation: 
\begin{equation}\label{integ}
u(t)={ \int_{0}^{1} G (t,s)g(s)ds, }
\end{equation}
where $G (t,s)$ is given in \eqref{eq:::1.2}. Moreover, $u$ in \eqref{integ} are in $AC[0,1]\cap E_\alpha $ and $D^{\alpha-1}_{0+}u$ is absolutely continuous in any compact subinterval of $(0,1)$.
\end{lemma}
 They proved that the solution $u$ may not be in $AC^2[0,1]$ so that we understand a solution $u$ is in $E_\alpha \cap AC[0,1] $ with $D^{\alpha-1}_{0+} u(t)$ which is absolutely continuous in any compact subinterval of (0,1) and $u$ satisfies the equation $(FDE_1)$ for almost everywhere $t\in [0,1]$ and boundary conditions
In this paper, we introduce some definitions and lemmas related to fractional calculus and Krasnoselski's classical fixed point theorem in Section 2. 
In Section 3, we set up a corresponding solution operator of problem $\ref{FDE}$ and
prove the existence of a positive solution for the problem.

\section{Preliminaries}

In this section, we introduce some definitions of fractional calculus and some important lemmas, and a theorem that will be used later.
\begin{definition}
We first introduce the basic Banach spaces 
\begin{itemize}
\item $AC[0,1]$ : the space of absolute continuous functions on $[0,1]$ 
\item $C^1_{\gamma}[0,1]= \{u\in C[0,1] : t^{\gamma} u'(t)\in C[0,1], u(0)= 0 = u(1) \}$ with the norm $\|u\|_{C^1_{\gamma}}=\|u\|_\infty+\|u'\|_{C_{\gamma}}$ where $0<\gamma<1$,  $\|u\|_{\infty}=\max_{t\in[0,1]}|u(t)|$ and $\|u\|_{C_{\gamma}}=\max_{t\in[0,1]}|t^{\gamma}u(t)| $
\item $E_\alpha= \{u\in C[0,1] : t^{\alpha-1} D^{\alpha-1}_{0+}u(t) \in C[0,1] \}$ equipped with the norm $\|u\|_{E_\alpha}=\|u\|_{\infty}+\|u\|_1$ where $\|u\|_1=\max_{t\in[0,1]} |t^{\alpha-1} D^{\alpha-1}_{0+}u(t)|$
\end{itemize}
\end{definition}
\begin{definition}{\upshape(\cite{Q}) }\label{def2}
The integral 
$$I^\alpha_{0+}u(t)=\frac{1}{\Gamma (\alpha)}\int_0^t \frac{u(s)}{(t-s)^{1-\alpha}}ds, ~t>0$$
where $\alpha >0$ is called the Riemann-Liouville fractional integral of order $\alpha$.
\end{definition}
\begin{definition}{\upshape(\cite{Q})}
For a function $u(t)$ given in the interval $[0,\infty )$, the expression 
$$D^\alpha_{0+}u(t)=\frac{1}{\Gamma(n-\alpha)}\Big{(}\frac{d}{dx}\Big{)}^n \int_0^t \frac{u(s)}{(t-s)^{\alpha-n+1}}ds$$
where $n=[\alpha]+1, [\alpha]$ denotes the integer part of number $\alpha,$ is called the Riemann-Liouville fractional derivative of order $\alpha$.
\end{definition}
\begin{remark}{\upshape(\cite{Q})} \upshape\label{rmk2.3}
We note for $\lambda>-1$,
$$D^{\alpha}_{0+}t^{\lambda}=\frac{\Gamma(\lambda+1)}{\Gamma(\lambda-\alpha+1)}t^{\lambda-\alpha}.$$
giving in particular $D^{\alpha}_{0+}t^{\alpha-m}=0$, $m=1,2,\cdots,N$, 
where $N$ is the smallest integer greater than or equal to $\alpha$.
\end{remark}
\begin{lemma}{\upshape(\cite{E}) }\label{lem2.4}
Assume that $u\in C(0,1)\cap L(0,1)$. For $\alpha>0$, $D^{\alpha}_{0+} u(t)=0$ has a unique solution
$$u(t)=c_1 t^{\alpha-1}+c_2 t^{\alpha-2}+\cdots +c_n t^{\alpha-n}, \quad c_i \in \mathbb{R}, i=1,2,\cdots,n $$
 where n is the smallest integer greater than or equal to $\alpha$.
\end{lemma}
As $D^{\alpha}_{0+}I^{\alpha}_{0+} u(t)=u(t)$ for all $u\in C(0,1)\cap L(0,1)$. From Lemma \ref{lem2.4}, we deduce the following statement.
\begin{lemma}{\upshape(\cite{E},\cite{O}) }
Assume that $u\in C(0,1)\cap L(0,1)$ with a fractional derivative of order $\alpha>0$ that belongs to $C(0,1)\cap L(0,1)$. Then
$$I^{\alpha}_{0+}D^{\alpha}_{0+} u(t)=u(t)+c_1 t^{\alpha-1}+c_2 t^{\alpha-2}+\cdots +c_n t^{\alpha-n}, \quad c_i \in \mathbb{R}, i=1,2,\cdots,n .$$
Moreover, if $0<\alpha <1$ and $u(t) \in C[0,1]$, then $D^{\alpha}_{0+} u(t) \in C(0,1)\cap L(0,1)$ and $$I^{\alpha}_{0+}D^{\alpha}_{0+} u(t)=u(t).$$
\end{lemma}
In the paper (\cite{E},\cite{Q}), the writers introduce useful properties of the Green's function $G(t,s)$ defined by \eqref{eq:::1.2} as follows:
\begin{lemma}{\upshape(\cite{E},\cite{Q}) }\label{lem2.5}
The Green function $G(t,s)$ satisfies the following conditions: 
\begin{itemize}
\item[$(1)$]\ $G(t,s) \in C([0,1] \times [0,1])$, and $G(t,s)>0$ for $t,s$ $\in (0,1),$
\item[$(2)$]\ $\max_{0\leq t \leq 1}G(t,s)=G(s,s), ~s\in (0,1)$
\item[$(3)$]\ $G(t,s)=G(1-s,1-t),$ for $t,s \in (0,1),$
\item[$(4)$]\ $  \frac{\alpha-1}{\Gamma(\alpha)}t^{\alpha-1}(1-t)(1-s)^{\alpha-1}s \leq G(t,s) \leq  \frac{1}{\Gamma(\alpha)}t^{\alpha-1}(1-t)(1-s)^{\alpha-2}.$
\item[$(5)$]\ $  G(t,s) \leq  \frac{1}{\Gamma(\alpha)}s(1-s)^{\alpha-1} t^{\alpha-2}.$
\end{itemize}
\end{lemma}
The property $(1), (2)$ are proved in Lemma 2.4 by monotonicity of $G(t,s)$ in the paper \cite{E}. And the authors in \cite{Q} deduced properties $(3)$ and $(4)$.
From the properties $(3)$ and $(4)$, we can have 
\begin{align}
G(t,s)=G(1-s,1-t)\le \frac{1}{\Gamma(\alpha)}(1-s)^{\alpha-1}(s)(t)^{\alpha-2}.
\end{align}
\begin{theorem} {\upshape (Fixed point theorem of cone expansion/compression type)}\label{thm2.4}
Let {$E$} be a Banach space and let {$K$} be a cone in {$E$}. Assume that {$\Omega_1$} and {$\Omega_2$} are open subsets of {$E$} with
{$0\in \Omega_1$}, {$\overline{\Omega_1}\subset \Omega_2$}. Assume that {$T: K\cap (\overline{\Omega_2}\setminus \Omega_1) \to K$}
is completely continuous such that either
\begin{itemize}
\item [$(1)$] \ {$\|Tu\|\leq \|u\|$}, for {$u\in K\cap \partial \Omega_1$} and
{$\|Tu\|\geq \|u\|$}, for {$u\in K\cap \partial\Omega_2$}, or
\item [$(2)$] \ {$\|Tu\|\geq \|u\|$}, for {$u\in K\cap \partial \Omega_1$} and
{$\|Tu\|\leq \|u\|$}, for {$u\in K\cap \partial\Omega_2$}.
\end{itemize}
Then {$T$} has a fixed point in {$K\cap ( \overline{\Omega_2}\setminus \Omega_1)$}.
\end{theorem}


\section{An Application to Nonlinear Problems}
Our goal in this section is to prove an existence result for following nonlinear problem
\begin{equation*}
\quad\quad\quad\quad\quad\quad\quad
\begin{cases}
 D^{\alpha}_{0+}u(t)+h(t)f(u(t))= 0,\quad t\in (0,1),\\
u(0)= 0 = u(1),
\end{cases}
\quad\quad\quad\quad\quad{(FD)}
\end{equation*}
where $D^{\alpha}_{0+}$ is the Riemann-Liouville fractional derivative of order $\alpha \in (1,2]$, 
$f\in C([0,\infty), [0,\infty))$ and $h$ satisfies $(H_1)$ and $(H_2)$.
In \cite{LL}, the researchers proved that a solution of 
$$u(t)=\int_0^1 G(t,s)h(s)f(u(s))ds:=Su(t),$$
where $G(t,s)$ is the Green's function given by \eqref{eq:::1.2} satisfies the equation $(FD)$. In this paper, we use the fixed point method to find a solution of the integral problem and this fixed point satisfies our equation $(FD)$.
We now state our main theorem in this section. 
For $u\in C[0,1]$, we define a nonlinear operator $S:C[0,1] \rightarrow C[0,1]$ by 
$$Su(t)=\int_0^1 G(t,s)h(s)f(u(s))ds.$$ 
Assume that $f$ satisfies (A1). Then, there exists a constant $p>1, r, R, \epsilon$ and $M_1$ such that $$|f(u)|<\epsilon |u|~~ \text{for}~ |u|<r,$$  $$|f(u)|<\epsilon |u|^p ~ \text{for}~ |u|>R,$$ and $$|f(u)|<M_1~ \text{for}~ r\le |u|\le R$$ 
 and then we have
\begin{align*}
|(Su)(t)|&\le \int_{0}^{1} G(t,s)h(s)f(u(s))ds\\
&\le \int_{0}^{1} \frac{s^{\alpha-1}(1-s)^{\alpha-1}h(s)}{\Gamma(\alpha)}f(u(s))ds\\
&<\int_{|u|<r} \frac{s^{\alpha-1}(1-s)^{\alpha-1}h(s)}{\Gamma(\alpha)}\epsilon |u(s)| ds+\int_{r<|u|<R} \frac{s^{\alpha-1}(1-s)^{\alpha-1}h(s)}{\Gamma(\alpha)}M_1 ds\\&+\int_{|u|>R} \frac{s^{\alpha-1}(1-s)^{\alpha-1}h(s)}{\Gamma(\alpha)}\epsilon |u(s)|^p ds\\
&<\int_0^1 \frac{s^{\alpha-1}(1-s)^{\alpha-1}h(s)}{\Gamma(\alpha)}ds\hat{M}<\infty 
\end{align*} 
where $\hat{M}=\epsilon \|u\|_\infty+M_1+\epsilon \|u\|^p_\infty$.
Moreover, by the similar calculation in the paper \cite{LL}, we have
\begin{align*}
D_{0+}^{\alpha-1}Su(t)&=\int_0^t ((1-\tau)^{\alpha-1}-1)h(\tau)f(u(\tau))d\tau-\int_t^1 (1-\tau)^{\alpha-1}h(\tau)f(u(\tau))d\tau
\end{align*} 
and then 
\begin{align*}
&|t^{\alpha -1}D_{0+}^{\alpha-1}Su(t)|\\&=t^{\alpha -1} \int_0^t (1-(1-\tau)^{\alpha-1})h(\tau)f(u(\tau))d\tau+t^{\alpha -1}\int_t^1 (1-\tau)^{\alpha-1}h(\tau)f(u(\tau))d\tau\\
&=t^{\alpha -1} \int_0^t [\int_{1-\tau}^1 (\alpha -1)s^{\alpha -2}ds] h(\tau)f(u(\tau))d\tau +t^{\alpha -1}\int_t^1 \frac{\tau^{\alpha-1}}{\tau^{\alpha-1}}(1-\tau)^{\alpha-1}h(\tau)f(u(\tau))d\tau\\
&\le t^{\alpha -1} \int_0^t (\alpha -1)(1-\tau)^{\alpha -2}(1-(1-\tau)) h(\tau)d\tau \hat{M}+t^{\alpha -1}\int_t^1 \frac{\tau^{\alpha-1}}{t^{\alpha-1}}(1-\tau)^{\alpha-1}h(\tau)d\tau \hat{M}\\
&\le [(\alpha -1) \int_0^t (1-\tau)^{\alpha -2}\tau h(\tau)d\tau +\int_t^1 \tau^{\alpha-1}(1-\tau)^{\alpha-1}h(\tau)d\tau ] \hat{M}
\end{align*} 
Since $((1-\tau)^{\alpha-1}-1)h(\tau), \tau h(\tau) \in L^1(0,1)$, we conclude that $Su :C[0,1] \rightarrow E_\alpha [0,1]$ is well defined and bounded when $f$ satisfies (A1). With the similar technique in the proof of Lemma 3.2 in \cite{LL}, it follows that $Su\in AC[0,1]\cap E_\alpha [0,1]$ and therefore $u=Su \in E_\alpha[0,1]$. Similarly, we get the same result if $f$ satisfies (A2). As a results, if we can find the fixed point of $u=Su$ in $C[0,1]$, it can be a solution of our equation $(FD)$ and then $u\in E_\alpha [0,1]$. So, we can get the following Theorem
\begin{theorem}\label{thm4.1}
Assume that the given function $h \in L^{1}_{loc}((0,1),[0,\infty))$ satisfies $(H_1)$, $(H_2)$ and $f$ satisfies either $(A1)$ or $(A2)$. Then problem $(FD)$ has at least one positive solution in $E_\alpha [0,1]$.
\end{theorem}

For this, let $E_\alpha$ be endowed with the ordering $u\leq v$ if $u(t)\leq v(t)$ for all $t \in [0,1]$. We define a cone $\mathcal{K}\subseteq E_\alpha$ by
$$\mathcal{K}=\{u \in C[0,1] \ | \ u(t)\geq 0, \ u(t)\geq(\alpha-1)t(1-t)\|u\|_\infty \}.$$ 
We first check the compactness of operator $S$. 


\begin{lemma}
Let $S:\mathcal{K} \to E_\alpha$ be the operator defined by
$$Su(t)=\int_{0}^{1} G(t,s)h(s)f(u(s))ds $$
where $f \in C( [0,\infty),[0,\infty))$  satisfying either (A1) or (A2) and $h\in L^1_{loc}((0,1),(0,\infty))$ satisfies $(H1)$ and $(H2)$.
Then $S:\mathcal{K} \to \mathcal{K}$ is completely continuous.
\end{lemma}
\begin{proof}
Since the Green's function, $h$ and $f$ are nonnegative, 
$$Su(t)=\int_{0}^{1} G(t,s)h(s)f(u(s))ds \geq 0.$$
By Lemma \ref{lem2.5}, we obtain
\begin{eqnarray*}
Su(t)&=&\int_{0}^{1} G(t,s)h(s)f(u(s))ds \\ &\geq & \frac{\alpha-1}{\Gamma(\alpha)}t^{\alpha-1}(1-t)\int_0^1 s(1-s)^{\alpha-1}h(s)f(u(s))ds.
\end{eqnarray*}
$$\|Su\|_\infty\leq \frac{1}{\Gamma(\alpha)}t^{\alpha-2}\int_0^1 s(1-s)^{\alpha-1}h(s)f(u(s))ds.$$
Hence, $$(\alpha-1)t(1-t)\|Su\|_\infty \leq Su(t).$$
This implies that $Su \in \mathcal{K}.$ 
 Let $\{u_k \} \subseteq E_\alpha$ be a convergent sequence to $u\in E_\alpha$. For any given $\epsilon>0,$ we let $$\epsilon_1=\frac{\epsilon}{\int_{0}^{1} G(s,s)h(s)ds}.$$ 
By the continuity of a function $f$, there exists $r>0$ such that for any $u_k$ with $|u-u_k|<r$, we have
\begin{align*}
|Su(t)-Su_k(t)|&\le \int_{0}^{1} |G(t,s)h(s)[f(u(s))-f(u_k(s))]|ds\\
&\le \int_{0}^{1} G(s,s)h(s)|f(u(s))-f(u_k(s))|ds\\
&\le \int_{0}^{1} G(s,s)h(s)ds\epsilon_1<\epsilon .
\end{align*}
By the continuity of $f$ and the condition $(H1), (H2)$, we can conclude that the operator $S$ is continuous. 
Let $\mathcal{M}$ be a bounded subset in $\mathcal{K}$.
Then  
we get
$$|Su(t)|\leq \int_{0}^{1} G(s,s)h(s)f(u(s))ds \leq M\int_{0}^{1} G(s,s)h(s)ds, $$
where $M=\sup_{u\in {\mathcal M}} \| f\circ u \|_{\infty}.$ 
Hence, $S(\mathcal M) $ is bounded by conditions $(H_1)$ and $(H_2)$.
\\
Now, we prove that $(Su)(\mathcal{M})$ is relatively compact subset of $C[0,1]$. Let $\{u_k\} \subseteq \mathcal{M}$.
It is proved that $Su' \in L^1(0,1)$ in Section 3 in \cite{LL}. From the fact that $f(u)$ is bounded, it follows that for any $0\le t_1<t_2\le 1,$
 \begin{align*}
|Su(t_2)-Su(t_1)|&\le  \int^{t_2}_{t_1}|Su'(s)|ds \rightarrow 0
\end{align*}
as $|t_1-t_2|\rightarrow 0$ for all $u\in \mathcal{M}$. Consequently, by Arzela-Ascoli theorem, $\{Su(t)\}_{u\in \mathcal{M}}$ is relatively compact. Therefore $S:\mathcal{K}\to \mathcal{K}$ is completely continuous and 
the proof is done.
\end{proof}
We now prove Theorem~\ref{thm4.1}. \\
\noindent
{\bf Proof of Theorem \ref{thm4.1}.}
Case 1. \ $f_0=0, f_\infty=\infty $.\\
By condition $f_0=0$, we may choose $r$ satisfying  for $0<u\leq r$, 
$$f(u)\leq  \varepsilon u,$$  
where $\frac{1}{2}>\varepsilon>0$ satisfies
$$\frac{\varepsilon}{\Gamma(\alpha)}\int_0^1 (s(1-s))^{\alpha-1}h(s)ds \leq  1.$$
Let $B_r =\{u\in  C[0,1] | ~\|u\|_{\infty}<r \}$, then for $u \in \mathcal{K}\cap\partial B_r$, we obtain
\begin{align*}
Su(t) &  \leq  \frac{1}{\Gamma(\alpha)}\int_0^1 (s(1-s))^{\alpha-1}h(s)f(u(s))ds\\
 & \leq \frac{1}{\Gamma(\alpha)} \int_0^1 (s(1-s))^{\alpha-1}h(s)\varepsilon u(s)ds\\
 & \leq \frac{\varepsilon}{\Gamma(\alpha)} \int_0^1 (s(1-s))^{\alpha-1}h(s)ds \|u\|_\infty \\
&\leq  \|u\|_\infty .
\end{align*}
Therefore, $\|Su\|_{\infty}\le \|u\|_{\infty}.$
Since $f_\infty=\infty$, there exists $M^{*}>0$ such that $f(u)\geq \rho u,$ for $u>M^{*}$, where $\rho >0$ is chosen so that  
$$\frac{\rho(\alpha-1)}{16}\int_{\frac{1}{4}}^{\frac{3}{4}} G(\frac{1}{2},s)h(s)ds\geq 1.$$
Let us take $R> \max\{\frac{16}{\alpha-1} M^{*}, r \}$ and
 $B_R =\{u\in  C[0,1] |\|u\|_{\infty}<R \}$, then for $u \in \mathcal{K}\cap\partial B_R$, 
 $$u(t)\geq \frac{\alpha-1}{16}\|u\|_\infty>M^{*},$$ 
for $t\in [\frac{1}{4},\frac{3}{4}]$ and
 \begin{align*}
\|Su\|_\infty \geq  Su(\frac{1}{2})&  =\int_0^1 G(\frac{1}{2},s)h(s)f(u(s))ds\\
 & \geq \int_{\frac{1}{4}}^{\frac{3}{4}} G(\frac{1}{2},s)h(s)f(u(s))ds\\
 & \geq \int_{\frac{1}{4}}^{\frac{3}{4}}  G(\frac{1}{2},s)h(s)\rho u(s)ds\\
 & \geq \frac{\rho(\alpha-1)}{16}\int_{\frac{1}{4}}^{\frac{3}{4}} G(\frac{1}{2},s)h(s)ds\|u\|_\infty\\
&\geq  \|u\|_\infty.
\end{align*}
 Therefore, by Theorem~\ref{thm2.4}, $S$ has a fixed point $u$ in  $u \in \mathcal{K}\cap ( \overline{B_R}\setminus B_r)$.\\ \\
Case 2.  \  $f_0=\infty$, $f_\infty=0$.\\
By condition $f_0=\infty$, we may choose $r_1$ so that $f(u)\geq L u$, for $0<u\leq r_1 $ where $L>0$ satisfies  
$$\frac{L(\alpha-1)}{16}\int_{\frac{1}{4}}^{\frac{3}{4}} G(\frac{1}{2},s)h(s)ds\geq 1.$$
Let $B_{r_{1}} =\{u\in  C [0,1] | \|u\|_{\infty}<r_{1} \}$, then for $u \in \mathcal{K}\cap\partial B_{r_{1}}$, we have
 \begin{align*}
\|Su\|_\infty \geq  Su(\frac{1}{2})&  =\int_0^1 G(\frac{1}{2},s)h(s)f(u(s))ds\\
 & \geq \int_{\frac{1}{4}}^{\frac{3}{4}} G(\frac{1}{2},s)h(s)f(u(s))ds\\
 & \geq \int_{\frac{1}{4}}^{\frac{3}{4}} G(\frac{1}{2},s)h(s)L u(s)ds\\
 & \geq \frac{L(\alpha-1)}{16}\int_{\frac{1}{4}}^{\frac{3}{4}} G(\frac{1}{2},s)h(s)ds\|u\|_\infty\\
&\geq  \|u\|_\infty.
\end{align*}
 Since $f_\infty=0$, there exists $L_2 >0$ such that $f(u)\leq  \zeta u$ ~for $u>L_2$ where $\zeta>0$ satisfies
$$\frac{\zeta}{\Gamma(\alpha)}\int_0^1 (s(1-s))^{\alpha-1}h(s)ds <  1.$$
And choose $R_2$ satisfying 
\begin{align*}
    R_2& > \max\{ L_2,\frac{\max_{0\leq u \leq L_2}|f(u)| \int_0^1 (s(1-s))^{\alpha-1}h(s)ds}{\Gamma(\alpha)-\zeta \int_0^1 (s(1-s))^{\alpha-1}h(s)ds}\}.
\end{align*}
Then we obtain
\begin{eqnarray*}
Su(t) &  \leq & \frac{1}{\Gamma(\alpha)}\int_0^1 (s(1-s))^{\alpha-1}h(s)f(u(s))ds\\
&  \leq&  \frac{1}{\Gamma(\alpha)}\Big[\int_{0\leq u \leq L_2} (s(1-s))^{\alpha-1}h(s)f(u(s))ds\\&&~~~~~~~~~~~~~~~~~+\int_{L_2\leq u \leq R_2} (s(1-s))^{\alpha-1}h(s)f(u(s))ds\Big]\\
&  \leq & \frac{1}{\Gamma(\alpha)}\Big[\max_{0\leq u \leq L_2}|f(u)|\int_{0\leq u \leq L_2} (s(1-s))^{\alpha-1}h(s)ds\\&&~~~~~~~~~~~~~~~~~+\int_{L_2\leq u \leq R_2} (s(1-s))^{\alpha-1}h(s)\zeta u(s)ds\Big]\\
&  \leq & \frac{1}{\Gamma(\alpha)}\Big(\max_{0\leq u \leq L_2}|f(u)|+\zeta \|u\|_\infty\Big)\int_{0}^{1} (s(1-s))^{\alpha-1}h(s)ds\\
&\leq &R_2= \|u\|_\infty .
\end{eqnarray*}
This implies that  $\|Su\|_{E_\alpha} \leq \|u\|_\infty$ for $u \in \mathcal{K}\cap\partial B_{R_2}$, and thus
$S$ has a fixed point $u$ in $\mathcal{K}\cap ( \overline{B_{R_2}}\setminus B_{r_1}).$
\begin{example}
In \cite{I}, Kong and Wang examined the existence of solution for the following equation
\begin{align}\label{(4.1)}
\begin{cases}
 D^{\alpha}_{0+}u(t)+\omega (t)u^{\theta}= 0,\quad 0<\theta<1<\alpha<2\\
u(0)= 0 = u(1),
\end{cases}
\end{align}
where $\omega \in$ C[0,1].
They showed that \eqref{(4.1)} has at least one positive solution.

Furthermore, if we take $f(t,u)=\omega(t)u^{\theta}$, then $f(t,t^{\alpha-2}u)=\omega(t)t^{(\alpha-2)\theta}u^{\theta}$. By choosing $g(u)=u^{\theta}$ and $q_{1}(t) = q_{2}(t)=\omega(t)t^{(\alpha-2)\theta},$ we get $q_{1},  q_{2} \in L(0,1)$ and $g_0=\infty$, $g_\infty=0$, thus $f$ satisfies condition (A) in \cite{Q} so that we get  the same result by applying Theorem 1.5  in \cite{Q} to  problem \eqref{(4.1)}.
When $\omega$ is given as $\omega(t)=t^{-\beta}, 1<\beta<\alpha$, then problem \eqref{(4.1)}  becomes
\begin{align}\label{(4.2)}
\begin{cases}
 D^{\alpha}_{0+}u(t)+t^{-\beta}u^{\theta}= 0,\quad 0<\theta<\alpha-1,1<\beta<\alpha,1<\alpha<2\\
u(0)= 0 = u(1).
\end{cases}
\end{align}
and the above results are not applicable since $\omega \notin C[0,1]$
and condition (A) is not valid. Indeed, if we take $f(t,u)=t^{-\beta}u^{\theta}$, then $f(t,t^{\alpha-2}u)=t^{-\beta+(\alpha-2)\theta}u^{\theta}$ in which $t$-term is not integrable by the fact  that $-\beta-\theta<-\beta+(\alpha-2)\theta<-\beta<-1$. Therefore it is impossible to find $q_1,q_2$ satisfying condition (A).

But $\omega(t)=t^{-\beta}, 1<\beta<\alpha$ satisfies $(H_1),(H_2)$ and $f(u)=u^{\theta}$ satisfies conditions $(A2)$. 
Therefore, we can apply Theorem~\ref{thm4.1} to guarantee the existence of at least one positive solution in $E_\alpha [0,1]$ for \eqref{(4.2)}.
\end{example}

\section*{Acknowledgment}
The authors express their gratitude to anonymous referees for their helpful suggestions which improved final version of this paper.

\section*{Funding}
This work was supported by the National Research Foundation of Korea, Grant funded by the Korea Government (MEST) (NRF2016R1D1A1B04931741).

\section*{Availability of data and materials}
Data sharing not applicable to this article as no datasets were generated or analyzed during the current study.

\section*{Competing interests}
The authors declare that there is no competing interests for this paper.

\section*{Author's contributions}
All authors have equally contributed in obtaining new results in this article and also
read and approved the final manuscript.

\section*{Author's information}
  Jinsil Lee, Department of Mathematics, University of Georgia,  Athens, GA 30606, USA.
   E-mail: jl74942@uga.edu\\
 Yong-Hoon Lee, Department of Mathematics, Pusan National University, Busan 46241, Republic of Korea.
    E-mail: yhlee@pusan.ac.kr

\section*{Publisher's Note}
Springer Nature remains neutral with regard to jurisdictional claims in published maps and institutional affiliations.
\bibliographystyle{plain}

\end{document}